\documentclass[preprint,12pt]{elsarticle}%
\usepackage{amsfonts}
\usepackage{amsmath}
\usepackage[T1]{fontenc}
\usepackage[utf8]{inputenc}
\usepackage{amssymb}
\usepackage{graphicx}%
\providecommand{\U}[1]{\protect\rule{.1in}{.1in}}
\makeatletter
\def\ps@pprintTitle{\let\@oddhead\@empty
\let\@evenhead\@empty
\let\@oddfoot\@empty
\let\@evenfoot\@oddfoot
}
\makeatother
\newtheorem{theorem}{Theorem}

\newtheorem{corollary}[theorem]{Corollary}

\newtheorem{lemma}[theorem]{Lemma}

\newenvironment{proof}[1][Proof]{\noindent\textbf{#1.} }{\ \rule{0.5em}{0.5em}}
\begin{document}
%
\begin{frontmatter}%


%

\title{Independent sets in chain cacti}%

%

\author{Jelena Sedlar}%
%

\address
{University of Split, Department of mathematics, Matice hrvatske 15, 21000 Split, Croatia}%
%

\begin{abstract}%

In this paper chain cacti are considered. First, for two specific classes of
chain cacti (orto-chains and meta-chains of cycles with $h$ vertices) the
recurrence relation for independence polynomial is derived. That recurrence
relation is then used in deriving explicit expressions for independence number
and number of maximum independent sets for such chains. Also, the recurrence
relation for total number of independent sets for such graphs is derived.
Finaly, the proof is provided that orto-chains and meta-chains are the only
extremal chain cacti with respect to total number of independent sets
(orto-chains minimal and meta-chains maximal).%

\end{abstract}%
%

\begin{keyword}%


Cactus graph \sep Chain cactus graph \sep Independence set \sep Independence polynomial



\MSC[2008] 05C69 \sep05C35%

\end{keyword}%
%

\end{frontmatter}%



\section{Introduction}

The notion of cactus graph first appeared in scientific literature in 1950's.
Then such graphs were called Husimi trees after the author of the paper which
motivated their introduction (\cite{Husimi1950}), and which was about cluster
integrals in statistical mechanics. The same graphs and their generalizations
also served as simplified models of real lattices (\cite{Monroe2004}), they
were useful in the theory of electrical and communication networks
(\cite{Zmazek2005}) and in chemistry (\cite{Zmazek2003}). Enumerative aspects
of such graphs were also studied in various papers (\cite{HararyN1953}%
,\cite{HararyU1953}), and finally summarized in classical monograph on graph
enumeration by Harary and Palmer (\cite{HararyP1973}).

Later, these graphs were named cactus graphs in mathematical literature. The
interest for such graphs has arisen again recently, since it was noted that
some NP-hard problems can be solved in polynomial time for that class of
graphs (\cite{Zmazek2003}). Chain cacti were also studied from aspect of
matchings (\cite{Farrell1987}, \cite{Farrell1998}, \cite{Farrell2000},
\cite{Farrell2000b}). The class of graphs very similar to cactus graphs, i.e.
block-cactus graphs, was studied recently too (\cite{Zverovic1998}). This
paper is motivated by the paper of T. Do\v{s}li\'{c} and F. M\aa l\o y
(\cite{Doslic2010}) in which they presented recurrence relations and/or
explicit formulas for various invariants related with matchings and
independent sets for two specific kind of chain hexagonal cacti. Also, they
provided proof that those two kinds of chain hexagonal cacti are extremal
among all chain hexagonal cacti with respect to total number of matchings in a
graph and with respect to total number of independent sets in graphs. In the
same paper they proposed generalizing those results for general chain cacti.
The proposed generalization for matchings has been presented in
\cite{BianKina2011}. In this paper the generalization of results by T.
Do\v{s}li\'{c} and F. M\aa l\o y for independent sets is presented, and
therefore this paper nicely supplements \cite{BianKina2011} in generalizing
\cite{Doslic2010}.

The present paper is organized as follows. In the section 'Preliminaries' we
introduce some basic notation, the notions about independence and the classes
of graphs with which we deal throughout the paper. Third section 'Main
results' is divided in three parts. The first part of main results is about
ortho-chains, which is special class of chain cacti for which independence
polynomials and some results about size and number of maximum independent sets
are provided. The second part of main results is about meta-chains, another
special class of chain graphs for which the same results as for ortho-chains
are provided. Finally, the third part of main results provides the proof of
extremality of ortho-chains and meta-chains with respect to total number of
independent sets.

\section{Preliminaries}

All graphs considered in this paper will be finite and simple. For a graph $G
$ we denote the set of its vertices with $V(G)$ (or just $V$)\ and the set of
its edges with $E=E(G)$ (or just $E$). Subgraph of $G$ induced by set of
vertices $V^{\prime}\subseteq V$ will be denoted with $G\left[  V^{\prime
}\right]  .$ For a vertex $v$ (or set of vertices $V^{\prime}\subseteq V$) of
$G$ we will denote with $G-v$ (or with $G-V^{\prime}$) subgraph of $G$ induced
by $V\backslash\left\{  v\right\}  $ (or by $V\backslash V^{\prime}$). For an
edge $e$ (or set of edges $E^{\prime}\subseteq E$) of $G$ we will denote with
$G-e$ (or with $G-E^{\prime}$) subgraph of $G$ obtained by deleting edge $e$
(or by deleting set of edges $E\backslash E^{\prime}$). For a vertex $v$ of
graph $G$ we will denote with $N[v]$ set of vertices consisting of $v$ and all
vertices of $G$ adjacent to $V$. We say that vertex $v$ of connected graph $G$
is cut vertex if $G-v$ is disconnected.

We say that set of vertices $S\subseteq V$ of graph $G$ is \emph{independent
set} if no two vertices of $S$ are adjacent in $G.$ \emph{Size} of an
independent set $S$ is number of vertices contained in $S.$ An independent set
of the largest possible size is called \emph{maximum independent set}. The
size of maximum independent set of graph $G$ is called \emph{independence
number} (or the \emph{stability number}) of graph $G$ and is denoted with
$\alpha(G).$ The \emph{independence polynomial} of graph $G$ is defined with
\[
i(G;x)=\sum_{k=0}^{\alpha(G)}\Psi_{k}(G)x^{k}%
\]
where $x$ is a formal variable and $\Psi_{k}(G)$ denotes number of independent
sets in graph $G$ of size $k$. Obviously, $\Psi_{0}(G)=1$ and $\Psi
_{1}(G)=\left\vert V\right\vert .$ Setting $x=1$ in $i(G;x)$ we obtain number
of all independent sets in graph $G$ and denote it with $\Psi(G)$. For the
notation simplicity sake, we will often write $i(G)$ instead of $i(G;x)$ where
it doesn't lead to confusion. The following properties of independence
polynomials are well known.

\begin{theorem}
\label{tm_I1}Let $G$ be a graph and $v$ its vertex. Then
\[
i(G;x)=i(G-u;x)+x\cdot i(G-N[u];x).
\]

\end{theorem}

\begin{theorem}
\label{tm_I2}Let $G$ be graph consisting of components $G_{1},G_{2}%
,\ldots,G_{k}.$ Then
\[
i(G;x)=i(G_{1};x)\cdot i(G_{2};x)\cdot\ldots\cdot i(G_{k};x).
\]

\end{theorem}

It is easily verified that for path $P_{n}$ and cycle $C_{n}$ holds
\begin{align*}
i(P_{n};x)  &  =\sum_{k=0}^{\left\lfloor (n+1)/2\right\rfloor }\binom
{n+1-k}{k}x^{k},\\
i(C_{n};x)  &  =\sum_{k=0}^{\left\lfloor n/2\right\rfloor }\frac{n}{n-k}%
\binom{n-k}{k}x^{k}.
\end{align*}

A \emph{cactus graph} is a connected graph in which each edge is contained in
at most one cycle, which means that each block of a cactus graph is an edge or
a cycle. A \emph{chain cactus} is cactus in which each block is a cycle which
contains at most two cut vertices and each cut vertex is contained in exactly
two cycles.\ The \emph{length} of chain cactus is number of cycles it consists
of. We say that a cycle is $h-$\emph{cycle} if it contains $h$ vertices. An
$h-$\emph{cactus} is a cactus graph in which every block is $h-$cycle. A
\emph{chain }$h-$\emph{cactus} is cactus which is $h-$cactus and chain cactus.

Let $A_{n}$ be a chain $h-$cactus. Cycles of $A_{n}$ are denoted with
$C^{(1)},\ldots,C^{(n)}$ by order on chain, vertices of cycle $C^{(i)}$ are
denoted with $v_{k}^{(i)}$ by order so that $v_{h}^{(i)}$ is cut vertex
contained in $C^{(i-1)}$ too. Vertices of $C^{(1)}$ are denoted by order so
that $v_{1}^{(1)}$ is cut vertex. Cut vertices therefore have two labels, but
that will not lead to confusion. (Note that completely analogous notation can
be introduced to general chain cacti. One only needs to introduce $h_{i}$ as
number of vertices on $C^{(i)}$).

Cycles $C^{(1)}$ and $C^{(n)}$ of $A_{n}$ are called \emph{end cycles},
otherwise cycles of $A_{n}$ are called \emph{internal cycles}. Note that
internal cycles contain exactly two cut vertices, while end cycles contain
exactly one cut vertex. We say that an internal cycle $C^{(j)}$ of $A_{n}$ is
in $k-$\emph{position} if its two cut vertices are on distance $k$, i.e. if
$v_{h}^{(j+1)}=v_{k}^{(j)}.$ For $1-$position and $2-$position we introduce
names \emph{ortho-} and \emph{meta- position}. If all internal cycles of chain
$h-$cactus are in ortho-position, then such cactus is called
\emph{ortho-chain} and is denoted with $O_{n}$. If all internal cycles of
chain $h-$cactus are in meta-position, then such cactus is called
\emph{meta-chain} and is denoted with $M_{n}$.

This notation in chain $h-$cacti (applied on $O_{n}$ and $M_{n}$) is
illustrated with Figure \ref{figure1} and Figure \ref{figure2}.%

\begin{figure}[h]%
\centering
\includegraphics[
natheight=2.392900in,
natwidth=5.020200in,
height=1.4633in,
width=3.0398in
]%
{Figure1.jpg}%
\caption{Notation in orto-chain.}%
\label{figure1}%
\end{figure}
%

\begin{figure}[h]%
\centering
\includegraphics[
natheight=2.320300in,
natwidth=4.753900in,
height=1.42in,
width=2.879in
]%
{Figure2.jpg}%
\caption{Notation in meta-chain.}%
\label{figure2}%
\end{figure}

With $A_{n-j}$ we will denote subcactus of $A_{n}$ induced by cycles
$C^{(1)},\ldots,C^{(n-j)}$. Graph with only one vertex is considered to be
chain cactus of length $0$ and is denoted with $A_{0}.$ Note that there is
only one $A_{0},$ $A_{1},$ $A_{2}$ (these cacti are considered to be both
ortho- and meta-).

\section{Main results}

For the independence polynomials of short chains holds:
\begin{align*}
i(O_{0})  &  =i(A_{0})=i(M_{0})=1+x,\\
i(O_{1})  &  =i(A_{1})=i(M_{1})=i(C_{h})=x\cdot i(P_{h-3})+i(P_{h-1}),\\
i(O_{2})  &  =i(A_{2})=i(M_{2})=x\cdot i(P_{h-3})^{2}+i(P_{h-1})^{2}.
\end{align*}
These results follow easily from Theorems \ref{tm_I1} and \ref{tm_I2}. Now we
want to provide formulae for longer chains, specifically for longer
ortho-chains and meta-chains.

\bigskip

\noindent\textit{Ortho-chains\medskip}

The recurrence relation for independence polynomials of longer ortho-chains is
given by the following theorem.

\begin{theorem}
\label{tm_O1}The independence polynomials of $O_{n}$, for $n\geq3$, satisfy
\[
i(O_{n})=x\cdot i(P_{h-3})^{2}\cdot i(O_{n-2})+i(P_{h-2})\cdot i(O_{n-1}).
\]

\end{theorem}

\begin{proof}
For the independence polynomial of $O_{n}$ ($n\geq3$) holds
\begin{equation}
i(O_{n})=x\cdot i(P_{h-3})^{2}\cdot i(O_{n-2}-v_{1}^{(n-2)})+i(P_{h-1})\cdot
i(O_{n-1}-v_{1}^{(n-1)}). \label{For_O}%
\end{equation}
Similarly, for the independence polynomial of $O_{n}-v_{1}^{(n)}$ holds
\begin{equation}
i(O_{n}-v_{1}^{(n)})=x\cdot i(P_{h-3})^{2}\cdot i(O_{n-2}-v_{1}^{(n-2)}%
)+i(P_{h-2})\cdot i(O_{n-1}-v_{1}^{(n-1)}). \label{For_O1}%
\end{equation}
Now, since right-hand size of (\ref{For_O}) is a linear combination of
expressions satisfying (\ref{For_O1}), that means that it also satisfies
(\ref{For_O1}). Since right-hand side and left-hand side of (\ref{For_O}) are
equal, that implies that $i(O_{n})$ satisfies (\ref{For_O1}) too, and that
proves the theorem.
\end{proof}

\bigskip

By setting $x=1$ to the recurrence relation from Theorem \ref{tm_O1} we can
obtain the recurrence relation for $\Psi(O_{n})$ in which coefficients would
be total number of independent sets in different paths. To be precise, we
obtain
\[
\Psi(O_{n})=\Psi(P_{h-3})^{2}\cdot\Psi(O_{n-2})+\Psi(P_{h-2})\cdot\Psi
(O_{n-1}).
\]
Given that
\[
\Psi(P_{n})=\frac{3F_{n}+L_{n}}{2}%
\]
where $F_{n}$ is fibonacci and $L_{n}$ lucas number, we obtain
\[
\Psi(O_{n})=\left(  \frac{3F_{h-3}+L_{h-3}}{2}\right)  ^{2}\cdot\Psi
(O_{n-2})+\frac{3F_{h-2}+L_{h-2}}{2}\cdot\Psi(O_{n-1}).
\]
Therefore, for a specific $n$ and $h$ we could calculate exact $\Psi(O_{n})$
from that recurrence relation. Now, we proceed to maximum independent set. We
will establish size and number of such sets for ortho-chains, i.e.
independence number $\alpha(O_{n})$ and number of maximum independent sets
$\Psi_{\alpha(O_{n})}\left(  O_{n}\right)  $.

\begin{theorem}
\label{tm_O2}The independence number of $O_{n}$, for $n\geq1,$ is
\[
\alpha(O_{n})=\left\{
\begin{tabular}
[c]{ll}%
$\frac{nh}{2}-\left\lfloor \frac{n-1}{2}\right\rfloor ,$ & for $h$ even,\\
$\frac{n\left(  h-1\right)  }{2},$ & for $h$ odd.
\end{tabular}
\right.
\]

\end{theorem}

\begin{proof}
For the sake of notation simplicity, let us denote $\alpha_{n}=\alpha(O_{n})$
and $p_{n}=\deg(i(P_{n})).$ In the case of even $h$, from independence
polynomials we obtain
\begin{align*}
\alpha_{0}  &  =1,\\
\alpha_{1}  &  =\max\left\{  1+p_{h-3},p_{h-1}\right\}  =\frac{h}{2},\\
\alpha_{2}  &  =\max\left\{  1+2\cdot p_{h-3},2\cdot p_{h-1}\right\}  =h.
\end{align*}
From recurrence relation of Theorem \ref{tm_O1} we obtain
\[
\alpha_{n}=\max\left\{  h-1+\alpha_{n-2},\text{ \ \ }\frac{h-2}{2}%
+\alpha_{n-1}\right\}  .
\]
The proof is now by induction on $n$. The proof for odd $h$ is analogous.
\end{proof}

\begin{theorem}
\label{tm_O3}The number of maximum independent sets in $O_{n}$, for $n\geq2$,
is
\[
\Psi_{\alpha(O_{n})}(O_{n})=\left\{
\begin{tabular}
[c]{ll}%
$1,$ & for $h$ even and $n=2k,$\\
$2+\frac{kh}{2},$ & for $h$ even and $n=2k+1,$\\
$\left(  \frac{h+1}{2}\right)  ^{2},$ & for $h$ odd.
\end{tabular}
\right.
\]

\end{theorem}

\begin{proof}
Let us first prove the result for even $h.$ For the degree of independence
polynomials holds
\[
\deg(i(P_{h-3}))=\deg(i(P_{h-2}))=\frac{h-2}{2},
\]
with leading coefficient of $i(P_{h-3})$ being $1.$ First, we want to
establish $\Psi_{\alpha(O_{2k})}(O_{2k})$ which is a leading coefficient in
$i(O_{2k})$. Let us recall that by Theorem \ref{tm_O1} polynomial $i(O_{2k})$
satisfies
\[
i(O_{2k})=x\cdot i(P_{h-3})^{2}\cdot i(O_{2k-2})+i(P_{h-2})\cdot i(O_{2k-1}).
\]
We know by Theorem \ref{tm_O2} that
\[
\deg(i(O_{2k}))=kh-k+1
\]
It is easily verified that
\begin{align*}
\deg\left(  x\cdot i(P_{h-3})^{2}\cdot i(O_{2k-2})\right)   &  =hk-k+1,\\
\deg\left(  i(P_{h-2})\cdot i(O_{2k-1})\right)   &  =kh-k
\end{align*}
These degrees imply $\Psi_{\alpha(O_{2k})}(O_{2k})=\Psi_{\alpha(O_{2k-2}%
)}(O_{2k-2})$ (since leading coefficient in $i(P_{h-3})$ equals $1$). Now,
from $\Psi_{\alpha(O_{2})}(O_{2})=1$ follows $\Psi_{\alpha(O_{2k})}%
(O_{2k})=1.$

Similar reasoning yields
\[
\Psi_{\alpha(O_{2k+1})}(O_{2k+1})=\Psi_{\alpha(O_{2k-1})}(O_{2k-1})+\frac
{1}{2}h
\]
which together with $\Psi_{\alpha(O_{1})}(O_{1})=2$ implies the result.

In the case of odd $h$ we obtain
\[
\Psi_{\alpha(O_{n})}(O_{n})=\Psi_{\alpha(O_{n-1})}(O_{n-1}),
\]
and then the result follows from $\Psi_{\alpha(O_{2})}(O_{2})=\left(
\frac{h+1}{2}\right)  ^{2}.$
\end{proof}

\bigskip

\noindent\textit{Meta-chains\medskip}

\noindent By similar reasoning one can obtain analogous results for meta-chains.

\begin{theorem}
\label{tm_M1}The independence polynomials of $M_{n}$, for $n\geq3$, satisfy%
\[
i(M_{n})=\alpha\cdot i(M_{n-1})-x^{2}\cdot\beta\cdot i(M_{n-2})
\]
where%
\begin{align*}
\alpha &  =x^{2}\cdot i(P_{h-5})+x\cdot\left(  i(P_{h-4})+2\cdot
i(P_{h-5})\right)  +i(P_{h-4}),\\
\beta &  =x\cdot i(P_{h-5})^{2}+i(P_{h-5})^{2}+i\left(  P_{h-4}\right)  \cdot
i(P_{h-5})-i\left(  P_{h-4}\right)  \cdot i\left(  P_{h-6}\right)  .
\end{align*}

\end{theorem}

\begin{proof}
For the sake of notation simplicity let
\begin{align*}
H_{n}  &  =M_{n}-v_{2}^{(n)},\\
K_{n}  &  =M_{n}-\left\{  v_{1}^{(n)},v_{2}^{(n)},v_{3}^{(n)}\right\}  ,\\
p_{n}  &  =i(P_{n}).
\end{align*}
Therefore, we have
\begin{align*}
i(M_{n})  &  =i(H_{n})+x\cdot i(K_{n}),\\
i(H_{n})  &  =i(M_{n-1})\cdot p_{h-4}+x\cdot i(H_{n-1})\cdot\left(
p_{h-4}+p_{h-5}+p_{h-5}\cdot x\right)  ,\\
i(K_{n})  &  =i(M_{n-1})\cdot p_{h-5}+x\cdot i(H_{n-1})\cdot p_{h-6}.
\end{align*}
Substituting $i(K_{n-1})$ and $i(H_{n-1})$ to $i(M_{n})$ we obtain
\begin{equation}
i(M_{n})=i(M_{n-1})\cdot\left(  p_{h-4}+x\cdot p_{h-5}\right)  +i(H_{n-1}%
)\cdot\left(  x\cdot\left(  p_{h-4}+p_{h-5}+p_{h-5}\cdot x\right)  +x^{2}\cdot
p_{h-6}\right)  \label{For_M1}%
\end{equation}
Substituting $i(H_{n-1})$ to obtained expression gives
\begin{align*}
i(M_{n})  &  =i(M_{n-1})\cdot\left(  p_{h-4}+x\cdot p_{h-5}\right)  +\\
&  +i(M_{n-2})\cdot p_{h-4}\cdot\left(  x\cdot\left(  p_{h-4}+p_{h-5}%
+p_{h-5}\cdot x\right)  +x^{2}\cdot p_{h-6}\right)  +\\
&  +x\cdot i(H_{n-2})\cdot\left(  p_{h-4}+p_{h-5}+p_{h-5}\cdot x\right)
\cdot\left(  x\cdot\left(  p_{h-4}+p_{h-5}+p_{h-5}\cdot x\right)  +x^{2}\cdot
p_{h-6}\right)  .
\end{align*}
Substituting $i(H_{n-2})$ from (\ref{For_M1}) to this expression proves the
claim of the theorem.
\end{proof}

\bigskip

By setting $x=1$ to the recurrence relation from Theorem \ref{tm_M1} we can
obtain the recurrence relation for $\Psi(M_{n})$ in which coefficients would
be total number of independent sets in different paths. Therefore, for a
specific $n$ and $h$ we could calculate exact $\Psi(M_{n})$ from that
recurrence relation. Now, we proceed to maximum independent set. We will
establish size and number of such sets for meta-chains, i.e. independence
number $\alpha(M_{n})$ and number of maximum independent sets $\Psi
_{\alpha(M_{n})}\left(  M_{n}\right)  $.

\begin{theorem}
\label{tm_M2}The independence number of $M_{n}$, for $n\geq1,$ is
\[
\alpha(M_{n})=n\cdot\left\lfloor \frac{h}{2}\right\rfloor .
\]

\end{theorem}

\begin{proof}
Let us consider set
\[
S=\left\{  v_{2k-1}^{(j)}:1\leq j\leq n,1\leq k\leq\left\lfloor \frac{h}%
{2}\right\rfloor \right\}  .
\]
This set is obviously independent on $M_{n}$ and each cycle $C^{(j)}$ contains
exactly $\left\lfloor \frac{h}{2}\right\rfloor $ vertices from $S.$ Since
cycle $C_{h}$ can contain at most $\left\lfloor \frac{h}{2}\right\rfloor $
independent vertices, it follows that $S$ is maximum independent set.
\end{proof}

\begin{theorem}
\label{tm_M3}The number of maximum independent sets in $M_{n}$, for $n\geq2$,
is
\[
\Psi_{\alpha(M_{n})}\left(  M_{n}\right)  =\left\{
\begin{tabular}
[c]{ll}%
$1,$ & for $h$ even,\\
$\left(  \frac{h-1}{2}\right)  ^{n-2}\cdot\left(  \frac{h+1}{2}\right)  ^{2},$
& for $h$ odd.
\end{tabular}
\right.
\]

\end{theorem}

\begin{proof}
Because of Theorem \ref{tm_M2}, maximum independent set $S$ in $M_{n}$ must
contain $\left\lfloor \frac{h}{2}\right\rfloor $ on each cycle. Note that
cycle $C_{h}$ contains at most $\left\lfloor \frac{h}{2}\right\rfloor $
independent vertices. If $S$ contained cut vertex, it would be counted in
$\left\lfloor \frac{h}{2}\right\rfloor $ independent vertices on two different
cycles, and consequently $S$ wouldn't be maximum independent set. Therefore,
maximum independent set $S$ on $M_{n}$ cannot contain cut vertices. Now note
the following: if $S$ is maximum independent set on $M_{n} $ then $v_{1}%
^{(j)}\in S$ for $j=2,\ldots,h-1$. We conclude
\begin{align*}
\Psi_{\alpha(M_{n})}  &  =\Psi_{\left\lfloor \frac{h}{2}\right\rfloor
-1}(P_{h-3})^{n-2}\cdot\Psi_{\left\lfloor \frac{h}{2}\right\rfloor }%
(P_{h-1})^{2}=\\
&  =\Psi_{\alpha(P_{h-3})}(P_{h-3})^{n-2}\cdot\Psi_{\alpha(P_{h-1})}%
(P_{h-1})^{2}.
\end{align*}
The claim now follows from the number of maximum independent sets on path.
\end{proof}

\bigskip

\noindent\textit{Extremality\medskip}

\noindent Now we want to establish extremal chain $h-$cacti with respect to
total number of independent sets. For that purpose we define relation
$\preceq$ on polynomials. Let $f(x)=\sum_{i=0}^{n}a_{i}x^{i}$ and
$g(x)=\sum_{i=0}^{n}b_{i}x^{i}$ be two polynomials. We say that $f\preceq g$
if $a_{i}\leq b_{i}$ for every $i=0,\ldots,n$. We say that $f\prec g$ if
$f\preceq g$ and $f\not =g$. Now, we need following two lemmas.

\begin{lemma}
\label{lm_E1}Let $A_{n}$ be a chain cactus of length $n\geq2$. For $2\leq
k\leq\left\lfloor \frac{h_{n}}{2}\right\rfloor $ holds
\[
i(A_{n}-v_{1}^{(n)})\prec i(A_{n}-v_{k}^{(n)}).
\]

\end{lemma}

\begin{proof}
Throughout the proof we will focus on $C^{(n)},$ so we will use notation
$C=C^{(n)}$, $h=h_{n}$ and $v_{i}=v_{i}^{(n)}.$ Also, we will denote
$G_{k}=A_{n}-v_{k}^{(n)}$. The claim of the lemma is now
\[
i(G_{1})\prec i(G_{k}).
\]
To prove it we need to prove
\begin{align*}
i(G_{1})  &  \preceq i(G_{k}),\\
i(G_{1})  &  \not =i(G_{k}).
\end{align*}
To prove
\[
i(G_{1})\preceq i(G_{k}),
\]
it is sufficient to prove that for every independent set $S$ on $G_{1}$ there
is corresponding (1) independent set $S^{\prime}$ on $G_{k}$ which is (2) of
the same size as $S$ such that (3) mapping $S\mapsto S^{\prime}$ is injection.

CASE I: $v_{k}\not \in S.$ Then we define $S^{\prime}=S.$ Obviously,
$S^{\prime}$ is well defined independent set on $G_{k}$ of the same size as
$S$, and this mapping is injection. Note that in this case $v_{1}%
\not \in S^{\prime}$ (since $v_{1}\not \in S$).

CASE\ II: $v_{k}\in S.$ Note that in this case $v_{k+1}\not \in S,$ and also
$v_{1}\not \in S$ (since $v_{1}\not \in G_{1}$). We define $S^{\prime}$ in the
following manner:
\[%
\begin{tabular}
[c]{ll}%
$v\in S^{\prime}\Longleftrightarrow v\in S$ & for $v\in V\backslash V(C),$\\
$v_{i}\in S^{\prime}\Longleftrightarrow v_{k+1-i}\in S$ & for $1\leq i\leq
k-1,$\\
$v_{i}\in S^{\prime}\Longleftrightarrow v_{i+1}\in S$ & for $k+1\leq i\leq
h-1.$%
\end{tabular}
\]
This construction is illustrated on Figure \ref{figure3}. First note that
$S^{\prime}$ is well defined set of vertices from $G_{k}$ since $v_{k}%
\not \in S^{\prime}$ by definition. Now, note that $S^{\prime}$ is independent
on $G_{k}-v_{h}$ since $S$ is independent on $G_{1}$. Since $v_{h}%
\not \in S^{\prime}$ by construction, it follows that $S^{\prime}$ is
independent on $G_{k}$ too. Furthermore, $S$ and $S^{\prime}$ are of the same
size since their cardinalities obviously coincide on $V\backslash V(C)$ and
$\left\{  v_{1},\ldots,v_{k}\right\}  .$ Also, because of $v_{k+1}\not \in S$
their cardinalities coincide on $\left\{  v_{k+1},\ldots,v_{h}\right\}  $ too.
Let us now show that $S\mapsto S^{\prime}$ is injection. If sets $S$ differ on
$V\backslash V(C)$ or $\{v_{2},\ldots,v_{k},v_{k+2},\ldots,v_{h}\}$ then
corresponding sets $S^{\prime}$ differ on $V\backslash V(C)$ or $\{v_{1}%
,\ldots,v_{k-1},v_{k+1},\ldots,v_{h-1}\}$ respectively. Also, sets $S$ can't
differ on $v_{k+1}$ since in this case $v_{k+1}\not \in S$. Hence $S\mapsto
S^{\prime}$ is injection and we have proved the claim for this case. Note that
in this case $v_{1}\in S^{\prime}$ (since $v_{k}\in S$).

We still have to prove overall injectivity (across the cases) of mapping
$S\mapsto S^{\prime}$. Let now $S_{1}$ be independent set from first case
($v_{k}\not \in S_{1}$) and $S_{2}$ be independent set from second case
($v_{k}\in S_{2}$). What remains to be proved is that $S_{1}^{\prime}%
\not =S_{2}^{\prime}.$ But we have noted that in the first case $v_{1}%
\not \in S_{1}^{\prime},$ and in the second case $v_{1}\in S_{2}^{\prime}.$
Hence, $S_{1}^{\prime}\not =S_{2}^{\prime}$ and we conclude that mapping
$S\mapsto S^{\prime}$ is overall injection from independent sets on $G_{1}$ to
independent sets on $G_{k}$ of the same size.

To prove that
\[
i(G_{1})\not =i(G_{k})
\]
it enough to find independent set $S^{\prime}$ on $G_{k}$ for which there is
no independent set $S$ on $G_{1}$ such that $S\mapsto S^{\prime}$. Let
$v_{j}^{(n-1)}$ be the cut vertex on $C^{(n-1)}$ of $A_{n}$ such that
$v_{j}^{(n-1)}=v_{h}^{(n)}$. Let us consider set $S^{\prime}=\left\{
v_{j+1}^{(n-1)},v_{1}^{(n)},v_{h-1}^{(n)}\right\}  $. We claim that
$S^{\prime}$ is such set. Suppose contrary, i.e. that there is independent set
$S$ on $G_{1}$ such that $S\mapsto S^{\prime}$. Then $S$ should be from second
case since $v_{1}^{(n)}\in S^{\prime}$. From construction of second case
follows that $S=\left\{  v_{j+1}^{(n-1)},v_{h}^{(n)}=v_{j}^{(n-1)},v_{k}%
^{(n)}\right\}  .$ But such $S$ is not independent because of edge
$e=v_{j+1}^{(n-1)}v_{j}^{(n-1)}.$ Therefore, we have contradiction.
\end{proof}

%

\begin{figure}[h]%
\centering
\includegraphics[
natheight=3.039800in,
natwidth=8.139600in,
height=1.8516in,
width=4.9113in
]%
{Figure3.jpg}%
\caption{With the proof of Lemma \ref{lm_E1}. Squared filled vertices are
certainly included in independent set, while squared empty vertices are
certainly excluded.}%
\label{figure3}%
\end{figure}

\begin{lemma}
\label{lm_E2}Let $A_{n}$ be a chain cactus of length $n\geq2.$ For $3\leq
k\leq\left\lfloor \frac{h_{n}}{2}\right\rfloor $ holds
\[
i(A_{n}-v_{k}^{(n)})\prec i(A_{n}-v_{2}^{(n)}).
\]

\end{lemma}

\begin{proof}
Again, throughout the proof we will focus on $C^{(n)},$ so we will use
notation $C=C^{(n)},$ $h=h_{n}$ and $v_{i}=v_{i}^{(n)}.$ Also, we will denote
$G_{k}=A_{n}-v_{k}^{(n)}$. The claim of the lemma is now
\[
i(G_{k})\prec i(G_{2}).
\]
To prove it we need to prove
\begin{align*}
i(G_{k})  &  \preceq i(G_{2}),\\
i(G_{k})  &  \not =i(G_{2}).
\end{align*}
To prove
\[
i(G_{k})\preceq i(G_{2}),
\]
it is sufficient to prove that for every independent set $S$ on $G_{k}$ there
is corresponding (1) independent set $S^{\prime}$ on $G_{2}$ which is (2) of
the same size as $S$ such that (3) mapping $S\mapsto S^{\prime}$ is injection.
Let $S$ be independent set on $G_{k}$. We distinguish three cases.

CASE\ I: $v_{2}\not \in S.$ Then we define $S^{\prime}=S.$ Obviously,
$S^{\prime}$ is well defined independent set on $G_{2}$ ($v_{2}\not \in
S^{\prime}$ since $v_{2}\not \in S$), $S$ and $S^{\prime}$ are of the same
size and $S\mapsto S^{\prime}$ is injection. Therefore, the claim is proved in
this case. Note that in this case $v_{k}\not \in S^{\prime}$ (since
$v_{k}\not \in S$).

CASE\ II: $v_{2}\in S$ and $v_{k+1}\not \in S$. Note that in this case
$v_{1}\not \in S$ and also $v_{k}\not \in S$ (since $v_{k}\not \in G_{k}$). We
define $S^{\prime}$ in the following manner:
\[%
\begin{tabular}
[c]{ll}%
$v\in S^{\prime}\Longleftrightarrow v\in S$ & for $v\in V\backslash V(C),$\\
$v_{i}\in S^{\prime}\Longleftrightarrow v_{k+2-i}\in S$ & for $3\leq i\leq
k,$\\
$v_{i}\in S^{\prime}\Longleftrightarrow v_{i}\in S$ & for $k+1\leq i\leq h$.
\end{tabular}
\
\]
This construction is illustrated on Figure \ref{figure4}. First, note that
$S^{\prime}$ is well defined set of vertices from $G_{2}$ since $v_{2}%
\not \in S^{\prime}$ by definition. Further, if we denote $e=v_{k}v_{k+1}$ we
can see that $S^{\prime}$ is independent $G_{2}-e$ since $S$ is independent on
$G_{k}$. Since $v_{k+1}\not \in S^{\prime}$ by construction (follows from
$v_{k+1}\not \in S$) we conclude that $S^{\prime}$ is independent on $G_{2}$
too. As for the size, sets $S$ and $S^{\prime}$ are of the same size because
for every vertex from $S$ there is by definition corresponding vertex in
$S^{\prime}$. Further, if sets $S$ differ on $V\backslash V(C)\cup\left\{
v_{k+1},\ldots,v_{h}\right\}  $ or $\{v_{2},\ldots,v_{k-1}\}$ then
corresponding sets $S^{\prime}$ differ on $V\backslash V(C)\cup\left\{
v_{k+1},\ldots,v_{h}\right\}  $ or $\{v_{3},\ldots,v_{k}\}$ respectively.
Also, sets $S$ can't differ on $v_{1}$ since $v_{1}\not \in S.$ We conclude
that mapping $S\rightarrow S^{\prime}$ of sets from this case is injection.
Hence, we have proved the claim in this case. Note that in this case
$v_{1}\not \in S^{\prime}$ and $v_{k}\in S^{\prime}.$

CASE III. $v_{2}\in S$ and $v_{k+1}\in S.$ We define $S^{\prime}$ in the
following manner:%
\[%
\begin{tabular}
[c]{ll}%
$v\in S^{\prime}\Longleftrightarrow v\in S$ & for $v\in V\backslash V(C),$\\
$v_{i}\in S^{\prime}\Longleftrightarrow v_{k+2-i}\in S$ & for $3\leq i\leq
k,$\\
$v_{i}\in S^{\prime}\Longleftrightarrow v_{h+k+2-i}\in S$ & for $k+2\leq i\leq
h,$\\
$v_{1}\in S^{\prime}.$ &
\end{tabular}
\]
This construction is illustrated on Figure \ref{figure5}. First note that
$S^{\prime}$ is well defined set of vertices from $G_{2}$ since $v_{2}%
\not \in S^{\prime}$ by definition. Now, let $e=v_{k}v_{k+1}$. Set $S^{\prime
}$ is independent set on $G_{2}-e-v_{h}$ since $S$ is independent on $G_{k}$.
Edge $e$ does not cause problems with independence of $S^{\prime}$, since
$v_{k+1}\not \in S^{\prime}$ by definition. Also, no edge incident with
$v_{h}$ is problem since $v_{h}\not \in S^{\prime}$ (since $v_{k+2}\not \in
S$, which is since $v_{k+1}\in S$). Therefore, $S^{\prime}$ is independent on
$G_{2}$ too. As for the size, sets $S$ and $S^{\prime}$ are of the same size
because for every vertex from $S$ there is by definition corresponding vertex
in $S^{\prime}$. Furthermore, if sets $S$ differ on $V\backslash V(C)$ or
$\{v_{2},\ldots,v_{k-1},v_{k+2},\ldots,v_{h}\}$ then sets $S^{\prime}$ differ
on $V\backslash V(C)$ or $\{v_{3},\ldots,v_{k},v_{k+2},\ldots,v_{h}\}$
respectively. Sets $S$ cannot differ on $v_{k+1}$ or $v_{1}$ since $v_{k+1}\in
S$ and $v_{1}\not \in S$ for all $S$ in this case. Therefore, mapping
$S\rightarrow S^{\prime}$ of sets from this case is injection. Hence, we have
proved the claim in this case too. Note that in this case $v_{1}\in S^{\prime
}$ and $v_{k}\in S^{\prime}$.

What remains to be proved is that mapping $S\rightarrow S^{\prime}$ is overall
injection (across the cases). Let $S_{1}$ be independent set from first case,
$S_{2}$ from second case and $S_{3}$ from third case. Than $S_{1}^{\prime
}\not =S_{2}^{\prime}$ and $S_{1}^{\prime}\not =S_{3}^{\prime} $ since
$v_{k}\not \in S_{1}^{\prime}$ and $v_{k}\in S_{2}^{\prime},S_{3}^{\prime}.$
Also, $S_{2}^{\prime}\not =S_{3}^{\prime}$ since $v_{1}\not \in S_{2}^{\prime
}$ and $v_{1}\in S_{3}^{\prime}.$ Therefore, mapping $S\rightarrow S^{\prime}$
is overall injection.

To prove that
\[
i(G_{k})\not =i(G_{k})
\]
it enough to find independent set $S^{\prime}$ on $G_{2}$ for which there is
no independent set $S$ on $G_{k}$ such that $S\mapsto S^{\prime}$. Let
$v_{j}^{(n-1)}$ be the cut vertex on $C^{(n-1)}$ of $A_{n}$ such that
$v_{j}^{(n-1)}=v_{h}^{(n)}$. Let us consider set $S^{\prime}=\left\{
v_{j+1}^{(n-1)},v_{1}^{(n)},v_{k}^{(n)},v_{k+2}^{(n)}\right\}  $. We claim
that $S^{\prime}$ is such set. Suppose contrary, i.e. that there is
independent set $S$ on $G_{1}$ such that $S\mapsto S^{\prime}$. Then $S$
should be from third case since $v_{k}^{(n)}\in S^{\prime}$ and $v_{1}%
^{(n)}\in S^{\prime}$. From construction of the third case follows that
$S=\left\{  v_{j+1}^{(n)},v_{h}^{(n)}=v_{j}^{(n-1)},v_{2}^{(n)},v_{k+1}%
^{(n)}\right\}  .$ But such $S$ is not independent because of edge
$e=v_{j+1}^{(n-1)}v_{j}^{(n-1)}.$ Therefore, we have contradiction.
\end{proof}

%

\begin{figure}[h]%
\centering
\includegraphics[
natheight=3.366700in,
natwidth=8.153500in,
height=2.047in,
width=4.9199in
]%
{Figure4.jpg}%
\caption{With the proof of Lemma \ref{lm_E2}. Squared filled vertices are
certainly included in independent set, while squared empty vertices are
certainly excluded.}%
\label{figure4}%
\end{figure}
%

\begin{figure}[h]%
\centering
\includegraphics[
natheight=3.119400in,
natwidth=8.206200in,
height=1.8991in,
width=4.9519in
]%
{Figure5.jpg}%
\caption{With the proof of Lemma \ref{lm_E2}. Squared filled vertices are
certainly included in independent set, while squared empty vertices are
certainly excluded.}%
\label{figure5}%
\end{figure}

Setting $x=1$ in polynomials from Lemmas \ref{lm_E1} and \ref{lm_E2} we obtain
following corollary.

\begin{corollary}
\label{cor_E3}Let $A_{n}$ be a chain cactus of length $n\geq2.$ For $3\leq
k\leq\left\lfloor \frac{h_{n}}{2}\right\rfloor $ holds
\[
\Psi(A_{n}-v_{1}^{(n)})<\Psi(A_{n}-v_{k}^{(n)})<\Psi(A_{n}-v_{2}^{(n)}).
\]

\end{corollary}

Now we can proceed with the main theorem.

\begin{theorem}
\label{tm_E4}Let $A_{n}$ be a chain $h-$cactus of length $n\geq3$ such that
$A_{n}\not =M_{n}$ and $A_{n}\not =O_{n}$. Then%
\[
\Psi(O_{n})<\Psi(A_{n})<\Psi(M_{n}).
\]

\end{theorem}

\begin{proof}
Let $A_{n}$ be any chain $h-$cacti of length $n\geq3$ such that $A_{n}%
\not =M_{n}$ and $A_{n}\not =O_{n}$. Then
\begin{align*}
i(A_{n})  &  =x\cdot i(A_{n-1}-N[v_{k}^{(n-1)}])\cdot i(P_{h-3})+i(A_{n-1}%
-v_{k}^{(n-1)})\cdot i(P_{h-1})=\\
&  =x\cdot i(A_{n-1}-N[v_{k}^{(n-1)}])\cdot i(P_{h-3})+i(A_{n-1}-v_{k}%
^{(n-1)})\cdot\left(  x\cdot i(P_{h-3})+i(P_{h-2})\right)  =\\
&  =x\cdot i(P_{h-3})\cdot\left(  i(A_{n-1}-N[v_{k}^{(n-1)}])+i(A_{n-1}%
-v_{k}^{(n-1)})\right)  +i(A_{n-1}-v_{k}^{(n-1)})\cdot i(P_{h-2}).
\end{align*}
Setting $x=1$ in these polynomials we obtain
\begin{equation}
\Psi(A_{n})=\Psi(P_{h-3})\cdot\Psi(A_{n-1})+\Psi(A_{n-1}-v_{k}^{(n-1)}%
)\cdot\Psi(P_{h-2}). \label{For_E4}%
\end{equation}
Note that the same holds for $O_{n}$ and $M_{n}$. Now we will prove by
induction on $n$ the following three claims simultaneously
\begin{align*}
\Psi(O_{n-1}-v_{1}^{(n-1)})  &  \leq\Psi(A_{n-1}-v_{1}^{(n-1)}),\\
\Psi(A_{n-1}-v_{2}^{(n-1)})  &  \leq\Psi(M_{n-1}-v_{2}^{(n-1)}),\\
\Psi(O_{n})  &  <\Psi(A_{n})<\Psi(M_{n}).
\end{align*}
For $n=3,$ the first two claims follow from $A_{2}=O_{2}=M_{2}$ and the third
claim follows from (\ref{For_E4}), the fact that $A_{2}=O_{2}=M_{2}$ and
Corollary \ref{cor_E3}.

For $n>3$, let us suppose that $C^{(n-2)}$ of $A_{n}$ is in $j-$position. Then
we have
\begin{align*}
\Psi(A_{n-1}-v_{1}^{(n-1)})  &  =\Psi(A_{n-2})\cdot\Psi(P_{h-3})+\Psi
(A_{n-2}-v_{j}^{(n-2)})\cdot\Psi(P_{h-4}).\\
\Psi(O_{n-1}-v_{1}^{(n-1)})  &  =\Psi(O_{n-2})\cdot\Psi(P_{h-3})+\Psi
(O_{n-2}-v_{1}^{(n-2)})\cdot\Psi(P_{h-4}).
\end{align*}
Since
\[
\Psi(A_{n-2})>\Psi(O_{n-2})
\]
by induction assumption, and also
\[
\Psi(A_{n-2}-v_{j}^{(n-2)})>\Psi(A_{n-2}-v_{1}^{(n-2)})\geq\Psi(O_{n-2}%
-v_{1}^{(n-2)})
\]
by Corollary \ref{cor_E3} and induction assumption respectively, we obtain
\begin{equation}
\Psi(O_{n-1}-v_{1}^{(n-1)})<\Psi(A_{n-1}-v_{1}^{(n-1)}). \label{For_E4a}%
\end{equation}
In a similar fashion we have
\begin{align*}
\Psi(A_{n-1}-v_{2}^{(n-1)})  &  =\Psi(A_{n-2})\cdot\Psi(P_{h-4})+\Psi
(A_{n-2}-v_{j}^{(n-2)})\cdot\left(  \Psi(P_{h-4})+2\Psi\left(  P_{h-5}\right)
\right)  ,\\
\Psi(M_{n-1}-v_{2}^{(n-1)})  &  =\Psi(M_{n-2})\cdot\Psi(P_{h-4})+\Psi
(M_{n-2}-v_{2}^{(n-2)})\cdot\left(  \Psi(P_{h-4})+2\Psi\left(  P_{h-5}\right)
\right)  .
\end{align*}
Since
\[
\Psi(A_{n-2})<\Psi(M_{n-2})
\]
by induction assumption and
\[
\Psi(A_{n-2}-v_{j}^{(n-2)})<\Psi(A_{n-2}-v_{2}^{(n-2)})\leq\Psi(M_{n-2}%
-v_{2}^{(n-2)})
\]
by Corollary \ref{cor_E3} and induction assumption respectively, we obtain
\begin{equation}
\Psi(A_{n-1}-v_{2}^{(n-1)})<\Psi(M_{n-1}-v_{2}^{(n-1)}). \label{For_E4b}%
\end{equation}
The inequality
\[
\Psi(O_{n})<\Psi(A_{n})<\Psi(M_{n})
\]
now follows from (\ref{For_E4}) since $\Psi(O_{n})<\Psi(A_{n})<\Psi(M_{n}) $
by induction assumption and since inequalities (\ref{For_E4a}) and
(\ref{For_E4b}) hold.
\end{proof}

\bigskip

Note that Lemmas \ref{lm_E1} and \ref{lm_E2} (and consequently Corollary
\ref{cor_E3}) hold for general chain cacti. Therefore, theorem for general
chain cacti, analogous to Theorem \ref{tm_E4}, can be proved. The only
condition is that number of vertices in $i-$th cycle (for $i=1,\ldots,n$) must
be the same for $A_{n},$ $O_{n}$ and $M_{n}$ of general case.

\begin{theorem}
Let $A_{n}$ be a chain cactus of length $n,$ and let $O_{n}$ and $M_{n}$ be
ortho- and meta- chain cacti of length $n$ such that $A_{n}$, $O_{n}$ and
$M_{n}$ have the same number of vertices on cycle $C^{(i)}$ for every
$i=1,\ldots,n$. If $A_{n}\not =O_{n}$ and $A_{n}\not =M_{n}$, then
\[
\Psi(O_{n})<\Psi(A_{n})<\Psi(M_{n}).
\]

\end{theorem}

\begin{proof}
Analogous to that of Theorem \ref{tm_E4}.
\end{proof}

\bigskip

For the end, we can propose some directions for further study. It would be
interesting to establish in what relation tree cacti stand to chain cacti with
respect to total number of independent sets.

\section{Acknowledgements}

Partial support of the Ministry of Science, Education and Sport of the
Republic of Croatia (grants. no. 083-0831510-1511) and of project Gregas is
gratefully acknowledged.

\end{document}